\documentclass[11pt]{amsart}
\usepackage{showkeys}
\usepackage{amsmath}
\usepackage{amsthm}
\usepackage{amssymb,amscd}

\newtheorem{thm}          {Theorem}      [section]
\newtheorem{prop}   [thm] {Proposition}
\newtheorem{lemma}  [thm] {Lemma}
\newtheorem{cor}    [thm] {Corollary}

\newtheorem{theorem}       [thm] {Theorem}

\theoremstyle{remark}

\def\CC{{\mathbb C}}

\def\fu{{\mathfrak{u}}}

\def\typea{{\mathfrak{sl}}}

\def\skipa{\vspace{-1.5mm} & \vspace{-1.5mm} & \vspace{-1.5mm}\\ }

\DeclareMathOperator\Sym{Sym}

\DeclareMathOperator\nsd{\circ}

\DeclareMathOperator\rank{rank}

\DeclareMathOperator\htt{ht}

\renewcommand\b\bar

\newcommand\bra[2]{{\langle#1,#2\rangle}}
\newcommand\paren[2]{{\left(#1,#2\right)}}

\newcommand\g{\mathfrak{g}}

\begin{document}

\title[A bound for highest weight modules] 
{A lower bound for the dimension of a highest weight module}
\author{Daniel Goldstein}
\address{Center for Communications Research,
San Diego, CA 92121, USA}
\email{danielgolds@gmail.com}
\author{Robert Guralnick}
\address{Department of Mathematics, University of Southern
California, Los Angeles, CA 90089-2532, USA}
\email{guralnic@usc.edu}
\author{Richard Stong}
\address{Center for Communications Research,
San Diego, CA 92121, USA}
\email{stong@ccrwest.org}

\thanks{We thank Nick Katz for asking for the question about representations with
dimension a product of two primes}.



\date{\today}

\begin{abstract}
For each integer $t>0$ and each complex simple Lie algebra $\g$,
we determine the least dimension of 
an irreducible highest weight representation of $\g$
whose highest weight has height $t$.
As a corollary, we classify all irreducible modules
whose dimension equals a product of two primes. 
\end{abstract}

\maketitle

\section{Introduction} \label{sec:intro}

Let $\g$ be a simple complex Lie algebra
of rank $n$. The irreducible finite-dimensional $\g$-modules
are the highest weight modules $V(\lambda)$ where 
$\lambda$ is a dominant weight.
Each $\lambda$ is uniquely a linear combination 
of $\lambda = \sum_{1\le i\le n}  a_i \lambda_i$ 
of the fundamental dominant weights $\lambda_i, 1 \le i \le n$, where
the $a_i(1\le i\le n)$ are non-negative integers. 
We define the {\it height} of $\lambda$, denoted 
$\htt(\lambda)$, to be $\htt(\lambda) = \sum a_i$. 


Our first main result determines, among all dominant weights 
$\lambda$ of a given height, the least value of $d=\dim V(\lambda)$ 
and gives a $\lambda$ achieving this minimum.

\begin{thm}  \label{main1} Let $\g$ be a 
simple complex Lie algebra of rank $n.$
There exists $1\le s\le n$ (depending on $\g$) such that 
$\dim V(\lambda) \ge \dim V(\htt(\lambda) \lambda_s)$
for all dominant weights $\lambda$.

Equality holds if and only if $\lambda = \htt(\lambda)\lambda_{s'}$ 
and there is an automorphism of the Dynkin digraph of $\g$ taking
$s$ to $s'$.
\end{thm}

See \cite{fh} for a basic reference.  We use the
Bourbaki notation for Dynkin diagrams, roots and weights;
see \cite[Planches II-IX, pp. 250--275]{bourbaki}.
Let $\Gamma$ be the finite group of graph automorphisms of the Dynkin
diagram.
Set $m=\min_{1\le i\le n} \dim V(\lambda_i)$.
The set of $s$ such that
$\dim V(\lambda_s)=m$
form a single $\Gamma$-orbit $S$.
This orbit has size $3$ in the case of $D_4$ ($s=1,3,4$), size
$2$ in the case of $A_n$ and $E_6$ ($s=1$ or $n$),
and size $1$ otherwise, in which case 
$s=n$ if $\g = F_4,E_7,E_8$ and $s=1$ for 
$B_n,n>2,C_n,n>1,D_n>4$.
(We choose to view $B_2$ as $C_2$).

If $\g$ is classical, then $V(t\lambda_s)$ has a nice
description in terms of symmetric powers
of the natural module.  We
give the formula for $\dim V(t \lambda_s)$ for 
exceptional $\g$ in  Table~\ref{table4} of Section~\ref{sec:cor}.
Note that, by the Weyl dimension formula, 
$f_j(t) = \dim V(t \lambda_j)$ is a polynomial in
$t$ of degree the dimension of the nilpotent radical of the
corresponding maximal parabolic subalgebra $\mathfrak{p}_j$ of $\g$
(indeed, a positive root $\alpha$ contributes to the product in the
numerator if and only if the coefficient of $\alpha_j$ in $\alpha$ is 
nonzero).
In particular, letting $t\to\infty$
in Lemma~\ref{lem2} we see the $s\in S$ are precisely those 
for which the dimension of the nilpotent radical of $\fu_{s}$ of the
maximal parabolic subalgebra corresponding to $\alpha_s$ is least.

We scale the Killing form $(,)$ on $\g^*\otimes g^*$ 
to be the 
unique invariant positive definite form
such that inner products of the root lattice elements are integral 
and have $\gcd = 1$.    Also,
by inspection, it follows that $(\lambda, 2\alpha)$ is 
integral for $\lambda$ in the weight lattice and $\alpha$ in
the root lattice.  Moreover, these inner
products are even if the type is not $B$.

Since $\rho = \lambda_1+\cdots + \lambda_n$ lies in the weight
lattice, the terms in the numerator (and denominator)
of the Weyl dimension formula
$(\lambda + \rho, \alpha)$ (and $(\rho, \alpha)$) 
are integers if $\g$ does not have type $B$.
The largest term in the numerator of the Weyl dimension formula
is  $(\lambda + \rho, \alpha_h)$,
where $\alpha_h$ denotes the highest positive root.

Thus, except for type $B$, 
$p \le (\lambda + \rho, \alpha_h)$, for any 
prime divisor $p$ of $\dim V(\lambda)$.

In type $B$, we have $p \le (\lambda + \rho, 2\beta)$ 
where $\beta$ is the highest short root. 

In this note, we classify  
the highest weight
representations $V(\lambda)$ such that 
\begin{equation} \label{ourbound}
\dim V(\lambda) \le (\lambda + \rho, \alpha_h)^2.
\end{equation}

\begin{thm} \label{main2} Let $\g$ be a simple complex Lie algebra
of rank $n\ge2$,
and $V(\lambda)$ a nontrivial highest weight module
module  for $\g$.
If   
\mbox{$\dim V(\lambda) \le (\lambda + \rho, \alpha_h)^2$}
then  $( g, \lambda)$ is in Table~\ref{table1} (if $n>2$) or
Section~\ref{rank2} (if $n=2$).
\end{thm}

In particular, we can now easily determine the irreducible
modules whose dimension is a product of at most two primes.
Gabber \cite[1.6]{katz} classifed those cases where $\dim V(\lambda)$ is
a prime and our methods give a somewhat different proof
of his result.
Inspired by a question of Nick Katz \cite[22.5]{katz3}, 
we prove the following
(see Table~\ref{table6} for more details):

\begin{cor} \label{pq}  Let $\g$ be a simple complex Lie algebra
of rank $n\ge2$
and $V(\lambda)$ an irreducible highest weight 
$\g$-module.  If $d:=\dim V(\lambda) = pq$ with $p$ and $q$
not necessarily distinct primes, then either $\g$ is classical and $V$
is the natural module or 
\begin{enumerate}
\item  $d  =a(a+2) $ with $a, a + 2$   prime;
\item  $d = a(2a-1)$ with $a, 2a - 1$   prime;
\item  $d = a(2a+1)$ with $a, 2a+ 1$   prime;
\item  $d = a(2a+3)$ with $a, 2a + 3$   prime;
\item  $d=  26, 77$ or $133$.
\end{enumerate}
\end{cor}

The methods of this paper cannot 
classify those modules of prime power dimension.  
If $s$ is any positive integer and 
$\lambda = (s-1) \sum \lambda_i$, then $\dim V(\lambda)=s^N$
where $N$ is the number of positive roots for $\g$ 
by the homogeneity of the dimension formula
as a polynomial in the variables $a_i +1$.   
Taking $s$ to be a prime power gives 
quite a lot of such modules.

In some special cases, the possibilities are quite limited.
For example, if $\g$ has rank at least $2$,
the only irreducible modules of dimension
$p^2$ for $p$ an odd prime are the natural module 
for $\g$ of type $A_{p^2 - 1}$ or $B_{(p^2-1)/2}$.
Similarly, the only irreducible modules of dimension
$2p$ with $p$ odd are either natural modules or  are of
dimension $6,10, 14$ or $26$. In particular, this gives:

\begin{cor} \label{pq skew} Let $\g =C_n$, $n$ an odd prime
with $V$ the natural module of dimension $2n$.
If $\mathfrak{h}$ is a proper Lie subalgebra of 
$\g$ that acts irreducibly on $V$, 
then $\mathfrak{h} = A_1$ or $n=7$ and $\g=C_3$.     
\end{cor}

\begin{cor} \label{pq nonselfdual} Let $\g$ be a simple 
complex Lie algebra of rank $n\ge2$
and $V(\lambda)$ a nontrivial irreducible highest weight $\g$-module
that is not self-dual. If $d:=\dim V(\lambda)$ is a product
of $2$ primes, then $\g=A_n$ and
one of the following holds:
\begin{enumerate}
\item $n > 3$ 
and $d=n+1, (n+1)(n+2)/2$ or $n(n+1)/2$;
\item $n=3,4$ or $6$ and $d=35$; or
\item $n=2$ and $d = a(a+1)/2$ or  $a(a+2)$.
\end{enumerate}
\end{cor}


In the first case, the modules are (up to duality),
the natural module, its symmetric square and its exterior
square.  The other possibilities can be read off Table \ref{table5}.
As we discuss in Section \ref{sec:cor}, one can determine whether the
module is self-dual and if so what type of form it preserves.

We remark that the same result holds if the characteristic is $\ell > 0$
as long as we add the assumptions that $\lambda$ is a restricted dominant
weight and $\dim V(\lambda) \le \ell$ (in characteristic $0$).
For then, it follows by 
\cite{jantzen}  that in this case 
the irreducible quotients of Verma modules have
the same dimensions in characteristics
$0$ and $\ell$.  
 
The paper is organized as follows.  In the next section, 
we  give the proof of Theorem \ref{main1}.  The proof involves
two steps.  We first show that if $\lambda$ has height $t$,
then $\dim V(\lambda) \ge \dim V(t \lambda_i)$ for some fundamental
weight $\lambda_i$.  We then give a combinatorial argument to 
show the minimum occurs only for the $\lambda_s$ as described above.

Once we have this result, if $\lambda$ has height $t$ and
$\dim V(t\lambda_s) \le \dim V(\lambda) \le (\lambda+\rho, \alpha_h)^2$,
we know that $\dim V(t \lambda_s)$ is a polynomial in $t$ of degree
the dimension of a nilpotent radical of a maximal parabolic and
$(\lambda+\rho, \alpha_h)^2$ is bounded by a quadratic polynomial in $t$.
If $\g$ has rank at least $2$ and $\g$ is not $A_2$, then 
$\dim V(t \lambda_s)$ is a polynomial of degree at least $3$ (and
typically much larger).  
Thus, excluding $\g = A_2$,  only finitely many $t$ are possible.   
Indeed, it follows as long the rank of $\g$ is not too small, 
we will see that $t \le 2$.   
In any case, for $t$ fixed, $\dim V(\lambda) \le (\rho + \alpha_h)^2$
will be quadratic in the rank.  
Moreover,  if $\g$ is classical of large enough rank, it is easy to see that if 
$\lambda$ involves $\lambda_i$ with $i \ge 3$
(and $i < n -2$ if $\g$ is of type $A$), then $\dim V(\lambda)$
will be at least cubic in the rank.  Thus, generically one only needs
to consider $V(a_1\lambda_1 + a_2 \lambda_2)$ with $a_1 + a_2 = t \le 2$.

We make these arguments precise in the following sections 
and get a complete
list of the modules satisfying our required bound.
We consider the rank $2$ Lie algebras
in a separate section.  Here one has to work harder and 
there are more
examples (indeed, infinitely many for $A_2$).

In Section
\ref{sec:cor}, we deduce the corollaries. 
We also recall some information about the type for self-dual modules
and formulas for the minimal dimension of modules of a given
height for  the exceptional algebras.

\section{Minimal dimension modules of a given height} \label{sec:min}

Theorem~\ref{main1} from follows from Lemmas~\ref{lem1} and~\ref{lem2}
Let $\g$ be a complex simple Lie algebra,
$\lambda = \sum a_i \lambda_i$ a dominant weight for $\g$, where
the $\lambda_i$ are the fundamental dominant weights.
We set $t=\sum a_i$.
Recall that $s$ is chosen so that 
$\dim V(\lambda_s)$ has
the smallest dimension among all fundamental domiminant weights.

\begin{lemma} \label{lem1}
 $\dim V(\lambda) \ge \min_{1\le j\le n}\dim V(t \lambda_j)$.
\end{lemma}

\begin{lemma} \label{lem2} For all $1\le j\le n$,
 $\dim V(t\lambda_j) \ge \dim V(t \lambda_s)$.
\end{lemma}

\begin{proof}[Proof of Lemma \ref{lem1}] 
Suppose that the linear function $m x + b$
takes positive values
on the real interval $I$  and set  $f(x)=\log(mx+b)$. 
Then  $f''(x)=-m^2/(mx+b)^2<0$,  so  $f$  is concave downward on $I$.

More generally, the logarithm of a product of linear functions,
each of which takes positive values on $I$, is concave downward.
Set, as usual, $\bra{\gamma}{\alpha} = 2 (\gamma,\alpha)/(\alpha,\alpha)$.
Consider the numerator of the Weyl dimension formula
$$
g(a_1,\cdots,a_n) = \prod_{\alpha>0}\bra{\rho+\lambda}{\alpha},
$$
where $\lambda=a_1\lambda_1+\cdots a_n\lambda_n.$
We can view $g$ as a function of the real orthant 
$a_1,\cdots ,a_n\ge0$,
where it takes values $\ge1$.
In particular, for any real $t\ge0$, 
$g$ takes positive values on the simplex 
$\Delta_t$  consisting of points
$P=(a_1,\cdots,a_n)$ in the orthant 
$a_1,\cdots,a_n\ge0$ that lie on the hyperplane $\sum a_i=t$.
Define the support of $P$ to be the number of $a_i$'s that
are nonzero.

It is enough to show that for any $P$ in $\Delta_t$ 
with support $\ge2$,
there is a $Q$ in $\Delta_t$ with strictly smaller support
such that $g(P) \ge g(Q)$.

Assume the support of $P$ is $\ge 2$.
Then there are two coordinates of $P$, 
say $a_i$ and $a_j$, that are nonzero.
Consider the line through $P$ in such that
all of the $a_1,\cdots, a_n$ except $a_i$ and $a_j$ are fixed,
and the sum $a_i+a_j$ is fixed.
Convexity implies that $g$ is minimized on $P\cap \Delta_t$ 
an an endpoint $Q$. Since either $a_i$ or $a_j=0$ at $Q$, 
the support of $Q$ is strictly less than the support of $P$.
\end{proof}

Let $\Phi^+$ be the set of positive roots of $\g$, and set, as usual,
$\rho = \frac12\sum_{\alpha\in\Phi^+} \alpha = \sum \lambda_i$.
The proof of Lemma~\ref{lem2} relies on the Weyl dimension
formula \cite[Cor. 24.6]{fh}:
$$ 
\dim V(\lambda) = \prod_{\alpha \in \Phi^+} 
\frac
{\bra{\rho +\lambda}{\alpha}}
{\bra{\rho }{\alpha}}.
$$

We wish to prove that $\dim V(t\lambda_j) \ge \dim V(t\lambda_s)$.
Since, when we apply the Weyl character formula, the denominators
are the same, it is enough to prove the inequality on the numerators:
$$
\prod_{\alpha \in \Phi^+}{\bra{\rho +\lambda_j}{\alpha}}
\ge\prod_{\alpha \in \Phi^+}{\bra{\rho +\lambda_s}{\alpha}}.
$$

Note that, for a positive root $\alpha$, 
$\bra{\lambda_i}{\alpha} > 0$ if and only if 
$\alpha$ occurs in the unipotent radical, $\fu_i$, 
of the maximal parabolic subalgebra determined by $\alpha_i$.
In particular, the function $f_i(t) = \dim V(t \lambda_i)$ is a polynomial
in $t$ of degree $\dim \fu_i$. 

Let $R_i$ denote the set of positive roots $\alpha$ which occur in $\fu_i.$
Our strategy for proving Lemma~\ref{lem2} is to find, for each $\g$
and each $1\le j\le n = \rank(\g)$, an injective 
function $\phi_j$ from $R_s$ to
$R_j$
such that (1) ${\bra{\rho}{\alpha}} \le {\bra{\rho}{\phi_j(\alpha)}}$
and (2) ${\bra{\lambda_s}{\alpha}} \le
{\bra{\lambda_j}{\phi_j(\alpha)}}$
for all $\alpha$ in $R_s$.
This strategy works in all cases except for $\g=B_n$ and $j=n$
when we do something more elaborate.

\begin{proof}[Proof of Lemma~\ref{lem2}]

Let  $\g$  be a simple complex Lie algebra of rank $n$.
Let  $\lambda_1,...,\lambda_n$ be the fundamental dominant weights of $\g$.

Recall that 
$\bra{\rho}{\alpha_j}=1$ and $\bra{\lambda_i}{\alpha_j}=\delta_{i,j}$.

We claim that there exists for each $1\le j\le n$ a function $\phi_j$
such that (1) and (2) are satisfied.
For the finitely many exceptional Lie algebras (types $E,F,G$) we verified
the claim (and therefore the lemma) by a straightforward computer search.
We proceed to prove the claim case by case for the classical Lie algebras 
(types $A,B,C,D$).

\medskip

\noindent Case $A_{n}, n\ge 2.$\\
The highest root is $\alpha_h = E_1-E_2$. We have
$\bra{\rho}{\alpha_h} = n-1.$ We choose $s=1$ rather than $s=n$.
Theunipotent radical $u_1$ corresponding to  $\alpha_1$ has dimension
$n$. 
The roots of $\fu_1$ are $b_1,\cdots b_{n}$ where
$b_i = e_1-e_{i+1}$. 
Subce all roots have squared-length $2$, we see that
$\bra{\lambda}{\alpha} = 2\paren{\lambda}{\alpha}/\paren{\alpha}{alpha} =
\lambda,\alpha)$.

Since $\bra{\lambda_s}{b_i}=(e_1,e_i-e_{i+1})=1 $ for all $b_i in \fu_1$, (2) is
automatically true.
Let $1\le j\le n$. We have $\alpha_j = e_j-e_{j+1}$ and 
the roots of $\fu_j$ are 
$\{e_i-e_k \mid 1\le i\le j < k\le n+1\}$.
Set $b'_1 = e_{j-1} - e_j, b'_2 = e_{j-2}-e_j\cdots, b'_{j-1}= b_j =
e_1-e_j$  and $b'_i = b_i$ for $i\ge j$.
We see that $\phi_j(b_i) = b'_i$ does the trick, since
$\bra{\rho}{b_i} = i = \bra{\rho}{b'_i}$.
\medskip 

\noindent Case $C_n, n\ge 2$.\\
Positive Roots $\Phi^+ = \{e_i \pm e_k \mid 1 \le i <k\le n\} \cup 
\{2e_i \mid 1 \le i \le n\}.$\\
Simple Roots $ e_1-e_2,\dots, e_{n-1}-e_n, 2e_n.$\\
Fundamental Dominant Weights $e_1, e_1+e_2,\cdots, e_1 + e_2+\cdots +e_{n-1}, 
e_1 + e_2+\cdots +e_n.$\\
$\rho = ne_1 + (n-1)e_2 + \cdots + e_n$,\\
We have $s=1$, and $\dim \fu_1 = 2n-1$.
The positive roots in $\fu_1$
are $b_i= e_1-e_{i+1}, (1\le i\le n-1),
b_n  =2e_i$, and
$b_{i} = e_1 + e_{2n+1-i} (n+1\le i \le 2n-1).$
With this notation, $\bra{\rho}{b_i} = i$ and 
$\bra{\lambda_1}{b_i}=1$ for all $1\le i\le 2n-1$,
and from the latter propert (2) is automatically true.

\noindent $j<n$. The simple root $\alpha_j = e_{j}-e_{j+1}$.\\
The $j'$th fundamental domninant weight is 
$\lambda_j =e_1+\cdots + e_j$ and the positive roots in 
$\fu_j$ are $\{e_i -  e_k \mid 1 \le i \le j < k \le n\}\cup 
\{e_i +  e_k \mid 1 \le i \le j \}
\cup \{2e_i \mid 1 \le i \le j\}.$

Set $b'_1 =  e_{j}-e_{j+1}, b'_2 =  e_{j-1}-e_{j+1}, \cdots, 
b'_{j} =  e_{1}-e_{j},$
and set $b'_i = b_i$ for $i>j$.
Then $\bra{\rho}{b'_i} = i = \bra{\rho}{b_i}$ for all $i$, hence (1) is true.

\noindent $j=n$. The simple root $2e_n$.\\
Here $\lambda_n= e_1 + e_2+\cdots +e_n.$
The positive roots in $\fu_n$ are 
$\{e_i +e_k \mid 1 \le i <k\le n\} \cup 
\{2e_i \mid 1 \le i \le n\}.$

Set $b'_1 = 2e_n,
b'_2 =  2e_{n-1}\cdots b'_n = 2e_1$ and 
$b'_i = b_i=e_i + e_{2n+1-i}$ for $n+1\le i < 2n-1$.
One checks that  $\bra{\rho}{b'_i} = i$ for all $i$, hence (1) is true.

\medskip 

\noindent Case $D_n, n\ge4$.\\
Positive Roots $\Phi^+ = \{e_i \pm e_j \mid 1 \le i <j\le n\}.$\\
Simple Roots $e_1-e_2,\dots, e_{n-1}-e_n, e_{n-1}+e_n.$\\
Fundamental Dominant Weights $e_1, e_1+e_2,\cdots, e_1 + e_2+\cdots
+e_{n-1},\frac12(e_1 + e_2+\cdots +e_n),\frac12(e_1 + e_2+\cdots
+e_n).$\\
We have $\rho = (n-\frac12)e_1 + (n-\frac32)e_2 + \cdots \frac12e_n$
We have $s=1$, and $\dim \fu_1 = 2n-2$.
The roots in $\fu_1$ are 
$b_i = e_1 - e_{i+1} (1\le i\le n-1)$, and 
$b_i = e_1 + e_{2n-i} ( n \le i \le n-2)$.
With this notation $\bra{\rho}{b_i} = i$ and
$\bra{\lambda_1}{b_i}=1$, so, from the latter, (2) is true.

\noindent $j<n$. The simple root $e_j-e_{j+1}.$\\
We have $\lambda_j = e_1 + \cdots e_j$ and 
the roots in $\fu_j$ are 
$\{e_i -  e_k \mid 1 \le i \le j<k\le n\}\cup
\{e_i +  e_k \mid 1 \le i \le j\}$
Set $b'_1 = e_{j-1}-e_{j}, b'_2 = e_{j-1} - e_j,\cdots e_j= e_1-e_j$
and $b'_i = b_i$  for $i \ge j$.
One checks that $\bra{\rho}{\b'_i} = i$, so (1) is true.

\noindent $j=n$. The simple root $e_{n-1}+e_n.$\\
We have $\lambda_n = \frac12(e_1 + \cdots e_n).$ 
The roots in $\fu_n$ are 
$\{e_i +  e_k \mid 1 \le i < k\le n\}.$
Set 
$b'_1 = e_{n-2}+e_{n-1}$ (and this is not equal to any other $b'_k$ 
since $n\ge 4$), 
$b'_k = e_{n-k}+e_{n}$ for $2\le k\le n-1$, and
$b'_k = b_k$ if $k \ge n$.
One checks $\bra{\rho}{b'_1} = 2>1=\bra{\rho}{b_1}$
and $\bra{\rho}{b'_i} = i = \bra{\rho}{b'_i}$ for $i>1$
so that (1) is true.

\medskip

\noindent Case $B_n, n\ge3$.\\
We have 
$\Phi^+ = \{e_i \pm e_j \mid 1\le i<j\le n\} \cup \{e_i \mid 1\le i\le
  n\}, \Delta = \{e_1-e_2,e_2-e_3,\cdots e_{n-1}-e_n, e_n\}$ and
fundamental dominant weights $\lambda_i = e_1 + e_2 + \cdots e_i (
1\le i <n)$ and $\lambda_n = \frac12(e_1 + e_2 + \cdots e_n)$, so that 
$\rho = (n-\frac12)e_1 + (n-\frac32)e_2 + \cdots + \frac12_n$.

We have $s=1$. We have $\dim(\fu_1) = 2n-1$, and the roots in $\fu_1$
are $b_i= e_1-e_{i+1}, (1\le i\le n-1), 
b_{i} = e_1 + e_{2n-i} (n\le i \le 2n-2)$ and $b_{2n-1} = e_1$.
With this notation, $\bra{\rho}{b_i} = i$ and 
$\bra{\lambda_1}{b_i} = 1$ except for $\bra{\lambda_1}{b_{2n-2}}=2$.

\noindent $j<n$. The simple root $e_j-e_{j+1}$.\\
Set $b'_1 =  e_{j-1}-e_{j}, b'_2 =  e_{j-2}-e_{j}, \cdots, 
b'_{j-1} =  e_{1}-e_{j},$
and set $b'_i = b_i$ for $i\ge j-1$.
We have $\bra{\rho}{b_i} = i = \bra{\rho}{b'_i}$ so that (1) is
satisfied. (2) only needs to be checked for the last root,
but then $\bra{\lambda_1}{b_{2n-1}} = 2 = \bra{\lambda_j}{b'_{2n-1}}$.
(This is exactly the same as the argument for $\g=C_n, j<n$.

\noindent $j=n$. The simple root $e_n$.\\
The roots in $\fu_n$ are  
$\{e_i + e_j \mid 1\le i<j\le n\} \cup \{e_i \mid 1\le i\le
  n\}.$
Set, for the odd subscripts,
$b'_1 = e_{n},b'_{3} = e_{n-1},\cdots b'_{2n-1} = e_1$, and
for the even subscripts,
$b'_2=e_{n-1} + e_n, b'_4 = e_{n-2}+ e_{n-1},\cdots b'_{2n-2} = e_1+e_2$.
With this notation, $\bra{\rho}{b'_i} = i$.
Both (1) and (2) are true for all roots but the last as 
$\bra{\lambda_1}{b_{2n-1}}  =2$ whilst $\bra{\lambda_1}{b'_{2n-1}}=1$.
Finally, set $b''_{2n-1}= e_1+e_n$ which is not one of the other
$b''_k$ since $k\ge3$. One checks using the following lemma
that the contribution from $b_{2n-1}$ to the Weyl dimension formula
for $V(\lambda_1)$ is not greater than the product of the
contributions to $V(\lambda_n)$ from $b'_{2n-1}$ and $b''_{2n-1}$.
This finishes the proof of the lemma.
\end{proof}

\begin{lemma} \label{lem:3roots}
If $t\ge0$  then 
$$
\frac{2t+(2n-1)}{2n-1} \le \frac{t+(2n-1)}{2n-1} \frac{t+n}{n} 
$$
\end{lemma}
\begin{proof}
Take the obvious inequality $n(2t+2n-1) \le (t+n)(t+2n-1)$
and divide both sides by $2n-1$.
\end{proof}

\section{Type A} \label{sec:A}

Let $\g=\typea_{n+1}, n > 2$ 
acting on the natural module $V(\lambda_1)=\CC^{n+1}$.
For any integer $t\ge0$, the module $V(t \lambda_1)$ is the
$t$'th symmetric power of the natural module.
Thus
\[ 
\dim V(t \lambda_1) =  \binom{n + t}{n}.
\]

Let $\lambda$ be a dominant weight for $\g$, and set
$t=\htt(\lambda)$.
The highest positive root is $\alpha_h = E_1 - E_{n+1}$, so that 
$(\rho, \alpha_h) = n$ and $(\lambda, \alpha_h) = t$.

We wish to classify all highest weights $\lambda$ such that satisfy 
\ref{ourbound}: 
$$
\dim V(\lambda) \le
(\rho + \lambda, \alpha_h)^2.
$$

For such a $\lambda$ we have,
by Theorem~\ref{main1},
$$
\binom{n + t}{n} = \dim V(t\lambda_1) \le \dim V(\lambda) \le 
(\rho + \lambda, \alpha_h)^2 = (n  + t)^2.
$$

Assume for the moment that $t \ge 3$.  
Then $\binom{n+t}{3}\le (n+t)^2$, which 
implies $n+t\le 8$.  
Thus $n \le 5$ and $\dim V(\lambda) \le 64$.  
An inspection of the Tables 6.6--6.9 
of \cite{lubeck} shows that only possibilities (up to graph automorphism)
for $(n, \lambda)$  are  $(n, 3 \lambda_1), 2 \le n \le 5$, 
$(n, 4\lambda_1), n =2,3$ and $(n, 5 \lambda_1), n=2,3$.

For the remainder of the section we assume that $t \le 2$. 

We will use the following result of L\"ubeck \cite[Theorem 5.1]{lubeck}
for groups of type $A$:

\begin{theorem} \label{lubeckthmA}  Let $\g=
\typea_{n+1}$ with $n > 11$.   Assume that 
$\dim V(\lambda) \le n^3/8 $, then (up to a graph automorphism) $\lambda = \lambda_1, 2 \lambda_1$
$\lambda_2$ or $\lambda_1 + \lambda_n$.  
\end{theorem}

One checks that, for any $n$, 
$\lambda_1, 2 \lambda_1, \lambda_2$,
and $ \lambda_1 + \lambda_n$ all satisfy the inequality.

If $n \le 11$,
Theorem~\ref{lubeckthmA} implies that $\dim V(\lambda) \le 169$.  Inspection of Tables 6.6--6.15
shows that the only additional dominant weights (up to duality) satisfying our inequality are
$(n, \lambda_3), 2 \le n \le 7$, $(3, \lambda_1 + \lambda_2)$ and $(3, 2\lambda_2)$.

Now assume that $n > 11$.  Then $\dim V(\lambda) \le (n+t)^2 = (n+2)^2
\le n^3/8$ and the result now follows by Theorem~\ref{lubeckthmA}.

\section{Type B} \label{sec:B}

Let $\g=B_n, n > 2$.  
Let $\CC^{2n+1} = V(\lambda_1)$ be the natural module for $\g$.

We note that in $\Sym^t(\CC^{2n+1})$, the submodule generated
by the multiples  the invariant quadratic form is invariant,
and its orthogonal complement is irreducible. In other words,
$V(t \lambda_1)\oplus \Sym^{t-2}(\CC^{2n+1}) = \Sym^t(\CC^{2n+1})$.

Thus, for $t > 1$,  
$\dim V(t \lambda_1) = 
\binom{2n+t+ 1}{2n+1} -  \binom{2n+t -1}{2n+1}.$

By Theorem \ref{main1}, we have $\dim V(\lambda) 
\ge \dim V(t \lambda_1)$ if $\lambda$ has height $t$.

The highest positive root is $\alpha_h = E_1+E_2$, so that
$(\rho, \alpha_h)=2n-2$ and  $(\lambda, \alpha_h)=2t - a_1-a_n\le 2t$,
where, as usual, $\lambda =\sum a_i\lambda_i$.

Thus $(\rho + \lambda, \alpha_h) \le 2n+2t-2.$

Note also that the highest short root is $\beta:=E_1$.
Note that $(\rho + \lambda, 2\beta) \le 2n + 2t - 1$.

Assume that $\dim V(\lambda) \le (\rho + \lambda, \alpha_h)^2$. Then 
for $t > 1$, 
$$
\binom{2n+t}{t} -  \binom{2n +t-2}{t-2} \le (2n+2t-1)^2.
$$

One checks that $ \lambda_1, 2\lambda_1,
\lambda_2$ and $\lambda_n,  n \le 9$ satisfy \eqref{ourbound}.

We see that the inequality implies $t\le 3$.
In particular, $\dim V(\lambda) \le (2n + 5)^2$. 

Following L\"ubeck \cite[Theorem 5.1]{lubeck}, we have:

\begin{theorem} \label{lubeckthm}  Let $\g$ be a classical Lie algebra
of type $B, C$, or $D$ of rank at $n >11$.  Assume that 
$\dim V(\lambda) \le n^3$, then $\lambda = \lambda_1, 2 \lambda_1$
or $\lambda_2$.
\end{theorem}

The basic point of the proof is to consider
orbits of the Weyl group on the weights.  
Although L\"ubeck states the result in positive characteristic,
his proof is equally valid  over the complex numbers.
(Alternatively, the result in positive characteristic {\it implies}
the result over the complex numbers. Indeed, for each $\lambda$,
for $p$ sufficiently large, $\dim V(\lambda)$ in characteristic $p$
equals its dimension over the complex numbers).

Thus, for $n > 11$, $\dim V(\lambda) < n^3$ and so $\lambda=\lambda_1, 
2 \lambda_1$ or $\lambda_2$.   If $n  \le 11$, this is a finite
problem and one can check precisely which modules satisfy the inequality
(either using the Weyl dimension formula or the tables in \cite{lubeck}
-- again, the tables in \cite{lubeck} are for positive characteristic but
for generic $p$, the dimension is the same as in characteristic $0$).  
We see that for $n \ge 4$, there are no further examples.

For $n=3$ we see that the  only further possibilities are 
$2 \lambda_3, 3\lambda_3, 3 \lambda_1,$ and $\lambda_1 + \lambda_3$
(of dimensions  $35, 112, 77$ and $48$). 

In order to deduce the result about the dimension being a product
of two primes, we note that a trivial consequence of the 
Weyl dimension formula is:

\begin{lemma} \label{type B}  Let $\g$ be of type $B_n, n \ge 3$.
If $p$ is a prime divisor of $\dim V(\lambda)$, then
$p \le (\rho + \lambda, 2\beta)$.
\end{lemma}

Thus, if $\dim V(\lambda)$ is a product of two primes, 
$\dim V(\lambda) \le (\rho + \lambda, 2\beta)^2$.  By
the classification, we see that it still follows 
that $\dim V(\lambda) \le (\rho + \lambda, \alpha_h)^2$.

\section{Type C} \label{sec:C}

Let $\g=C_n, n >  2$.  
In this case $V(\lambda_1) = \CC^{2n}$ is the natural module for $\g$.
It is known that $V(t \lambda_1)$
is the $t$'th symmetric power of the natural module.

By Theorem~\ref{main1}, if 
$\lambda$ is a dominant weight for $\g$
then 
$$   
\dim V(\lambda) \ge \dim V(t \lambda_1) =\binom{2n+t-1}{2n-1},
$$
where $t=\htt(\lambda)$.
 
Since the highest root for $\g$ is $\alpha_h = 2E_1$,
we see that 
$(\rho , \alpha_h) = 2n$
and 
$(\lambda, \alpha_h) = 2t,$ so that 
$(\rho + \lambda, \alpha_h) = 2n + 2t$. 

Assume that $\lambda$ is a dominant weight for $\g$ that satisfies
\eqref{ourbound}. We see that this  implies the inequality.
$$
\binom{2n+t-1}{2n-1} \le  (2n + 2t)^2.
$$

We see that $\lambda_1, \lambda_2$ and $2\lambda_1$
satisfy \eqref{ourbound}.

For $n\ge 6$, the inequality implies $t\le 2,$
whence $\dim V(\lambda) \le (2n+4)^2$.   By Theorem \ref{lubeckthm},
this implies that for $n \ge 11$, there are no further examples.
For $n < 11$, we check directly.  We see that in fact for
$n \ge 6$, there are no further examples.

If $n\le 5$, we have the additional possibility $\lambda_3$.
For $n\le 4$, the inequality implies $t\le 3$, and 
one checks that for $n=5$ there are no further possibilities.
If $n\le 4$, we have the additional possibilities
$\lambda_3$ and  $\lambda_4$ (if $n=4$). 
For $n=4$ there are no further possibilities.

Finally consider $C_3$.  The inequality implies
$t \le 5$. 
This leads to the further possibilities
$\lambda_1 + \lambda_2, 3 \lambda_1$.


\section{Type D}  \label{sec:D}

We consider $G=D_n, n \ge  4$ with natural 
module $\CC^{2n} = V(\lambda_1)$.
Just as for type  $B$, we have
$V(t \lambda_1)\oplus \Sym^{t-2}V(\lambda_1) = \Sym^tV(\lambda_1)$.
It follows that
$\dim V(t \lambda_1)= \binom{2n+t-1}{2n-1} - \binom{2n+t-3}{2n-1}$. 

The highest positive root is $\alpha_h = E_1+E_2$.
If $\lambda = \sum a_i \lambda_i$ is a dominant
weight, then
$(\rho, \alpha_h) = 2n-3 
\le 2n + 2t -3$.
$(\lambda , \alpha_h) = 2t - (a_1 + a_{n-1} + a_n)$ so that 
$(\lambda + \rho, \alpha_h) = 2n-3 + 2t - (a_1 + a_{n-1} + a_n)
\le 2n + 2t -3$.
 
Assume now that \eqref{ourbound} is true. Then we have the inequality
$$ \binom{2n+t-1}{2n-1} - \binom{2n+t-3}{2n-1} \le (2n + 2t -3)^2.
$$
  
One checks that $\lambda_1, 2 \lambda_1,$ and
$\lambda_2$ satisfy \eqref{ourbound}.
One also checks that the  half-spin representation 
$\lambda_n$ works for $n\le 9$.

If $n\ge5$ the inequality implies $t\le 2$, whence
$\dim V(\lambda) \le (2n+1)^2$.  As in the previous
cases, this implies by Theorem \ref{lubeckthm} that for
$n \ge 11$, there are no further examples.  One checks
for $5 \le n \le 10$, there are no other examples as well.

If $n=4$, the inequality implies $t \le 4$.
A computer check finds no further examples (up to the
(large) group of graph automorphisms).   

\section{Rank $2$ Lie algebras} \label{rank2}

Let $\g$ be a rank $2$ Lie algebra.
Let $\lambda =  a\lambda_1 + b\lambda_2$ be
a dominant weight for $\g$, where  $\lambda_1$ and $\lambda_2$
are fundamental dominant weights for $\g$ and $a,b\ge0$.
If $\dim V(\lambda)$ is a product of two (not necessarily distinct)
primes, then necessarily
$$
\dim V(\lambda) \le (\rho +\lambda, \alpha_h)^2.
$$
We also find all dominant weights such that the
inequality holds, and among them we identify those such that
$\dim V(\lambda) =pq$.

\subsection{$A_2$} \label{A2}
The inequality to be solved is $(a+1)(b+1)(a+b+2)/2 \le
(a+b+2)^2.$ This is equivalent to $(a+1)(b+1) \le 2(a+b+2)$, or
$(a-1)(b-1) \le 4$. The solutions are:
$a=0, b$ any, $a=1,b$ any $b=0,a$ any, $b=1,a$ any,
$a=2,b\le 4, b=2,a\le 4,$ or $a=b=3$.
The cases where  $\dim V(\lambda)=pq$ are (up to interchanging
$a$ and $b$):
If $u$ and $2u+1$ are prime: $\dim V((2u-1)\lambda_1)=u(2u+1)$.
If $u$ and $2u-1$ are prime:  $\dim V((2u-2)\lambda_1)=u(2u-1)$.
If $u$ and $u+2$ are prime: $\dim V((u-1) \lambda_1+ \lambda_2) = 
u(u+2)$.

\subsection{$C_2$}
Let $\lambda = a\lambda_1+ b\lambda_2$ be  a dominant
weight for $\g=C_2.$.
Then $\dim V(\lambda) = (a+1)(2b+2)(a+2b+3)(2a+2b+4)/24$ and
$(\rho+\lambda, \alpha_h) = (2a+2b+4)$.
The inequality to be solved,
$(a+1)(2b+2)(a+2b+3)(2a+2b+4)/24 \le (2a+2b+4)^2,$ is equivalent
to 
$$
(a+1)(b+1)(a+2b) \le 24 (a+b+2).
$$ 
This has $53$ solutions:
$$
\begin{array}{cc}
a=0 \ \rm{and}& b\le 11,\\ 
a=1 \ \rm{and} &b\le 5,\\
a=2 \ \rm{and} & b\le 3,\\
b=0\ \rm{and} & a\le22,\\
b=1 \ \rm{and} &  a\le9,\\
b=2 \ \rm{and} & a\le 5,\\
a=3 \ \rm{and} & b=3.
\end{array}
$$
Of these $7$ have dimension  $pq,$  namely
$\dim V(\lambda_1) = 4, 
\dim V(2\lambda_1) = 10, 
\dim V(2\lambda_2) =14, 
\dim V(2\lambda_1+\lambda_2) =\dim V(4\lambda_1) = 35,
\dim V(4\lambda_2) = 55,$ and $\dim V(5\lambda_2) = 91.$

\subsection{$G_2$}  
Let $\lambda = a\lambda_1+ b\lambda_2$ be  a dominant
weight for $\g=G_2.$. Then 
$
\dim V(\lambda)  =  (a+1)(b+1)(a+b+2)(a+2b+3)(a+3b+4)(2a+3b+5)/120$,
and $(\rho+\lambda, \alpha_h) = 2a+3b+5$.
The inequality  to be solved is
$(a+1)(b+1)(a+b+2)(a+2b+3)(a+3b+4) \le 120(2a+3b+5).$

We claim that the inequality implies $a,b\le 7$.
If $b\ge a$, then the left hand side is at least $(a+1)^4$,
and the right hand side is at most $120\cdot 5(a+1)$, whence
$b\le a\le 7$. There is a similar argument assuming $a\le b$.

One checks that the inequality has the $7$ solutions $$(a,b) =
(0,0),(0,1),(0,2),(1,0),(1,1),(1,2), (1,3).$$
Of these $3$ have dimension $pq$, namely:
$\dim V(\lambda_2) = 14$, and $\dim V(2\lambda_2) = 
\dim V(3\lambda_1) = 77$.

\section{Exceptional Lie Algebras}  \label{sec:EFG}
  
\subsection{$F_4$} 
 
In this section, we prove:

\begin{prop} \label{f4}  Let $\g=F_4$.  If $\lambda \ne0$ and
$\dim V(\lambda) \le (\rho + \lambda, \alpha_h)^2$,
then $V(\lambda)$ is the either the minimal module
of dimension $26$ or the adjoint module of dimension
$52$.
\end{prop}

Assume that $\lambda = \sum a_i \lambda_i$ with
$t = \sum a_i$ and $\dim V(\lambda) \le  
(\rho+\lambda,\alpha_h)^2$.  Note that
$(\rho+\lambda,\alpha_h) = 
(2a_1+3a_2 + 4a_3 + 2a_4  + 16) \le 4t + 11$. 

Let $f(t)$ be the function given in Table~\ref{table5} for $F_4$.
Then  
$\dim V(\lambda) \le f(t)/f(0)$. 
 It is straightforward to see that this implies
that $t \le 2$, whence $\dim V(\lambda) \le  19^2$.
The only additional $\lambda$ that satisfy this bound
are $\lambda_3$ and $2 \lambda_4$ (of dimensions
$273$ and $324$ respectively).  These do not satisfy
the required inequality, whence the result.

\subsection{$E_n, n = 6, 7, 8$}

First consider $\g = E_6$.

\begin{prop} \label{e6}  Let $\g=E_6$.  If $\lambda\ne0$ and
$\dim V(\lambda) \le (\rho + \lambda, \alpha_h)^2$,
then $V(\lambda)$ is either one  the two minimal modules
of dimension $27$ or the adjoint module of dimension
$78$. \end{prop}

\begin{proof}
Let $\lambda = \sum a_i \lambda_i$ with $t = \sum a_i$.
Let $f(t)$ be the polynomial given in Table for $E_6$.
Thus, $\dim V(\lambda) \le f(t)/f(0)$.  Note that 
$(\rho+\lambda,\alpha_h) = 
a_1+2a_2 + 2a_3 + 3a_4 +2a_5+a_6 + 11
\le  5t+11$. It is straightforward to see that this forces
$t \le 2$ and in particular, $\dim V(\lambda) \le 21^2$.
Inspection of the tables in \cite{lubeck} shows then implies
the result.
\end{proof} 
 
The same proof (only easier since the polynomial $f(t)$
has larger degree) gives the following for $E_7$ and $E_8$.

\begin{prop} \label{e7}  Let $\g=E_7$.  If $\lambda\ne0$ and
$\dim V(\lambda) \le (\rho + \lambda, \alpha_h)^2$,
then $V(\lambda)$ is either the minimal module
of dimension $56$ or the adjoint module of dimension
$133$.
\end{prop}

\begin{prop} \label{e8}  Let $\g=E_8$.  If $\lambda\ne0$ and
$\dim V(\lambda) \le (\rho + \lambda, \alpha_h)^2$,
$V(\lambda)$ is the adjoint module of dimension $248$.
\end{prop}

\section{The corollaries} \label{sec:cor}

\begin{table}[htbp]
\caption{Nonzero $\lambda$ such that $\dim V(\lambda) 
\le (\rho + \lambda,\alpha_h)^2$ and rank $n\ge3$
(up to graph automorphism)}
\label{table1}
\[
\begin{array}{|l|c|l|}
\hline
\g & \lambda & \dim V(\lambda)   \\
\skipa \hline \hline
A_n &  \lambda_1  &  n +1  \\ 
           &  2\lambda_1 & (n+1)(n+2)/2 \\ 
           &  \lambda_2 &  n(n+1)/2  \\ 
           &  \lambda_1 + \lambda_n & n(n+2)   \\
3\le n\le 7 &     \lambda_3          &   \binom{n+1}{3}  \\ 
3\le n\le 5 &    3 \lambda_1         &   \binom{n+3}{3}  \\ 
A_3  & t\lambda_1,  t =4,5 &  35,56  \\  
     &  \lambda_1 + \lambda_2 &  20  \\
     &   2\lambda_2           &   45 \\ \hline

B_n &  \lambda_1 &  2n+1  \\    
           &   2 \lambda_1 & n(2n+3)  \\ 
           &   \lambda_2   & n(2n+1)  \\ 
3\le n\le 9     &   \lambda_n     & 2^n   \\
B_3        &  2 \lambda_3 &  35      \\      
           & \lambda_1 + \lambda_3 &  48  \\ 
           & 3 \lambda_1        &    77   \\
           & 3 \lambda_3    &  112 \\ \hline

C_n        &   \lambda_1 &   2n  \\ 
           &   2\lambda_1 & n(2n+1) \\
           &   \lambda_2    & (n-1)(2n+1)  \\ 
3\le n\le 5     &   3\lambda_1 & \binom{2n+2}{3}  \\ 
           &    \lambda_3   & \binom{2n}{3} - 2n \\
C_4        &    \lambda_4  & 42    \\
C_3        &    \lambda_1 + \lambda_2 & 64  \\ \hline

D_n (4\le n)&   \lambda_1 &  2n  \\ 
           &   2\lambda_1 &  (2n-1)(n-1)  \\ 
           &   \lambda_2 &  n(2n-1)  \\ 
4\le n \le 9    &  \lambda_n  & 2^{n-1} \\ \hline
E_6        &   \lambda_1 &  27  \\ 
           &    \lambda_2     &  78 \\ \hline
E_7        &   \lambda_7 &  56  \\ 
           &    \lambda_1      &  133  \\ \hline
E_8        &    \lambda_8 & 248   \\ \hline
F_4        &    \lambda_4 &  26 \\ 
           &     \lambda_1 &  52 \\ \hline 
\end{array}
\]
\end{table}

It follows by the Weyl dimension formula that aside from
type B, 
the largest prime dividing the dimension of $\dim V(\lambda)$
is at most $(\lambda + \rho, \alpha_h)$.  Thus
if $\dim V(\lambda)$ is a product of at most $2$ primes, it
follows that $\dim V(\lambda) \le (\lambda + \rho, \alpha_h)^2$.

Thus, Corollaries \ref{pq}, \ref{pq skew} and \ref{pq nonselfdual}
follow immediately from Theorem \ref{main2} except for type B.

If $\g$ has type $B$, then the Weyl dimension formula implies
that the largest prime dividing the dimension of $\dim V(\lambda)$
is at most $(\lambda + \rho, \beta)$, where $\beta$ is the 
highest short root. Thus, if 
$\dim V(\lambda)$ is a product of at most $2$ primes, 
$\dim V(\lambda) \le (\lambda + \rho, \beta)^2$.
In Section \ref{sec:B}, we showed that this implied
that $\dim V(\lambda) \le (\lambda + \rho, \alpha_h)^2$
and so again the Corollaries follow from Theorem \ref{main2}.

We summarize in the table below when the irreducible highest weight
module $V(\lambda)$ is self-dual, and, in the positive case,
whether the module is symplectic or orthogonal.
For convenience,  define $2k+1$ to be the largest 
odd integer~$\le n$.
The first column of Table~\ref{table5} lists 
the simple complex Lie algebra,
the second column gives necessary and sufficient conditions
for $V(\lambda)$ to be self-dual, the third
gives necessary and sufficient conditions for the self-dual
representation  $V(\lambda)$ to be symplectic.

\begin{table}[htbp]
\caption{Duality.}
\label{table5}
\[
\begin{array}{|l|c|c|}
\hline
\g & V(\lambda)\rm{\ Self-Dual?}& \rm{Symplectic?} \\
\hline\hline
A_n, n\ge1 & \rm{For\ all\ } 1\le i\le n, 
a_i = a_{n+1-i} & n\equiv 1\!\!\!\mod 4 \rm{\ and\ }
a_{k+1} \rm{ \ odd\ }\\
B_n, n\ge3 & \rm{always} & n\equiv 1\rm{\ or\ }2 \!\!\!\mod 4 \rm{\ and\ }
a_{n} \rm{ \ odd\ }\\ 
C_n, n\ge2 & \rm{always}  & a_1+a_3 + \cdots +a_{2k+1} \rm{\ odd\ }\\ 
D_n, n\ge4 \rm{\ even\ } & \rm{always} & 
n\equiv 2\!\!\!\mod 4 \rm{\ and\ } a_{n-1}+a_n \rm{\ odd\ }\\ 
D_n, n\ge5 \rm{\ odd\ } & a_{n-1}=a_n &  \rm{\ never\ }\\ 
E_6 & a_1=a_6 \rm{\ and\ } a_2=a_5& \rm{\ never\ }\\
E_7 &\rm{always} & a_2+a_5+a_7 \rm{\ odd\ }\\
E_8 & \rm{always}& \rm{\ never\ }\\
G_2 & \rm{always}& \rm{\ never\ }\\
F_4& \rm{always}& \rm{\ never\ }\\
\hline
\end{array}
\]
\end{table}

It is easy to tell whether an irreducible module is self-dual.
All modules are if there is no graph automorphism or for type
$D_n$ with $n$ even.   In the other cases,  $V(\lambda)$
is self-dual if and only if it is invariant under the graph
automorphism, where $\lambda = \sum_{i=1}^na_i\lambda_i$ 
is a dominant weight.

If a module is self-dual, then it either supports a non-zero
invariant alternating form (symplectic case)
or a non-zero invariant symmetric form (orthogonal case) but
not both.
See, for example, \cite[5.1.24]{gw} for the well-known
criterion to distinguish these cases.
To distinguish the orthogonal from the symplectic cases,
it is sufficient to do so for the fundamental dominant weights.  
Indeed, let $B$ denote the set of fundamental dominant weights
$\lambda_i$
such that $V(\lambda_i)$ is (self-dual and) symplectic.
Then the self-dual highest weight module $V(\lambda)$ is 
symplectic if and only if $\sum_{i\in B} a_i$ is odd.
We note that by \cite[Lemmas 78, 79]{steinberg} the criterion
holds for simple algebraic groups in positive characteristic $p \ne 2$
as well.  

We list all symplectic fundamental dominant weights. 

\begin{table}[htbp]
\caption{Fundamental Dominant Weight Modules {$V(\lambda_i)$}
that are Symplectic}
\label{table3}
\[
\begin{array}{|l|c|}
\hline
\g &  i  \\
\hline \hline
A_n, n \equiv 1 \mod 4     &   (n+1)/2  \\ \hline
B_n, n \equiv 1, 2 \mod 4  &   n \\ \hline
C_n                      &    i \ \text{odd} \\ \hline
D_n, n \equiv 2 \mod 4     &   n-1, n  \\ \hline
E_7                       &  2,5,7  \\ \hline
\end{array}
\]
\end{table}

\newpage 
\begin{table}[htbp]
\caption{This table gives  a monic polynomial $f(t)$ such that $\dim V(t\lambda_s) =f(t)/f(0)
$ for $\g$ exceptional. Note that $\dim V(t\lambda_s)$ is the least
dimension
among all highest weight modules $V(\lambda)$ of height
$\htt(\lambda)=t$.}
\label{table4}
\[
\begin{array}{|l|c|c|}
\hline
\g &  f(t)  & \rm{degree(f)}\\\hline

G_2  &  (t+1) (t+2) (t+3) (t+4) (2t+5)  & 5 \\
F_4  &\prod_1^{10}(t+j)\prod_{4}^{7}(t+j)\cdot(2t+11)
     & 15 \\
E_6  &    \prod_{1}^{11} (t+j) 
       \prod_{4}^{8} (t+j)  
     &  16              \\ 
E_7  &  \prod_{1}^{17}(t+j)\prod_{5}^{13}(t+9)
     & 27 \\ 
E_8  & \prod_{1}^{28}(t+j)\prod_{6}^{23}(t+j)
       \prod_{10}^{19}(t+j)\cdot(2t+29)
     & 57
\\ \hline
\end{array}
\]
\end{table}

\normalsize
\begin{table}[htbp] \label{table6}
\caption{\protect{$\dim V(\lambda)=pq$ with $p$ and $q$ prime}}
\[
\begin{array}{|l|c|l|l|}
\hline
pq & \g & \lambda  & \text{Duality} \\ 
\hline\hline
\rm{any} & A_1 &  (pq-1)\lambda_1 & \circ\\
& A_{pq-1} & \lambda_1 & -\\
& C_{pq/2}\ (pq \ \rm{even}) & \lambda_1 & -\\
& B_{(pq-1)/2}\  (pq \ \rm{odd}) & \lambda_1 & +\\
& D_{pq/2}\ (pq \ \rm{even}) & \lambda_1 & +\\

\hline
a(2a+1)
& A_{2a-1}   & 2\lambda_1  & \circ \\ 
& A_{2a}     & \lambda_2   & \circ \\
& B_a (a>2)  & \lambda_2   & +\\
& C_a        & 2 \lambda_1 & +\\
& A_2        & (2a-1)\lambda_1 & \circ\\
& D_{a+1}    & 2 \lambda_1 & +\\
\hline
a(2a-1)
&A_{2a-2} & 2\lambda_1      & \circ\\
&A_{2a-1}(a>2) & \lambda_2       & \circ\\
&A_3 & \lambda_2       & +\\
&A_2      & (2a-2)\lambda_1 & \circ\\
&D_a      & \lambda_2      & +\\
\hline
a(a+2)
&A_a      &  \lambda_1+\lambda_n       & +\\
&A_2      & (a-1)\lambda_1 + \lambda_2 & \circ\\
\hline
a(2a+3)
&B_a   & 2\lambda_1 & +\\
&C_{a+1} & \lambda_2 & +\\
\hline
14
&C_2 & 2\lambda_2 & +\\
&C_3 & \lambda_3 & - \\
&G_2 & \lambda_2 & +\\
\hline
26
&F_4  & \lambda_4 & +\\
\hline
35
&A_3 & 4\lambda_1 & \circ\\
&A_4 & 3\lambda_1 & \circ\\
&A_6 & \lambda_3  &\circ\\
&B_3 & 2\lambda_3 &+\\
&C_2 & 2\lambda_1+\lambda_2 & +\\
&C_2 & 4\lambda_1 & +\\
\hline
55
&C_2 & 4 \lambda_2 & + \\
\hline
77
&B_3  & 3\lambda_1 & +\\
&G_2  & 2\lambda_2 &+\\
&G_2  & 3\lambda_1 &+\\
\hline
91
&C_2 & 5\lambda_2 & +\\
\hline
133
&E_7  & \lambda_1  & +\\
\hline
\end{array}
\]
\end{table}

\end{document}